\documentclass[10pt]{amsart}
\usepackage{amsmath,amssymb,amsfonts, amscd,enumerate}
\usepackage{amsbsy}
\usepackage{amsthm}
\usepackage{amstext}
\usepackage{amsopn}
\usepackage{marvosym}
\usepackage{mathrsfs}
\usepackage{graphicx}
\usepackage{latexsym}
\usepackage[T1]{fontenc}
\usepackage{color}
\usepackage{tikz}

\newcommand{\Hmm}[1]{\leavevmode{\marginpar{\tiny%
$\hbox to 0mm{\hspace*{-0.5mm}$\leftarrow$\hss}%
\vcenter{\vrule depth 0.1mm height 0.1mm width \the\marginparwidth}%
\hbox to 0mm{\hss$\rightarrow$\hspace*{-0.5mm}}$\\\relax\raggedright #1}}}

\newcommand{\Cc}{\mathcal{C}}
\newcommand{\cC}{\mathcal{C}}
\newcommand{\Dc}{\mathcal{D}}
\newcommand{\Ec}{\mathcal{E}}

\newcommand{\Vc}{\mathcal{V}}

\def \rmi{{\rm i}}

\newcommand{\Er}{\mathscr{E}}
\newcommand{\Fr}{\mathscr{F}}
\newcommand{\Gr}{\mathscr{G}}

\newcommand{\Kr}{\mathscr{K}}

\newcommand{\Tr}{\mathscr{T}}

\newcommand{\Ur}{\mathscr{U}}
\newcommand{\Vr}{\mathscr{V}}
\newcommand{\Wr}{\mathscr{W}}

\newcommand{\ve}{\varepsilon}

\def \bone{\mathbf{1}}

\newtheorem{thm}{Theorem}[section]
\newtheorem{cor}[thm]{Corollary}
\newtheorem{lem}[thm]{Lemma}
\newtheorem{lemma}[thm]{Lemma}
\newtheorem{pro}[thm]{Proposition}

\theoremstyle{definition}
\newtheorem*{defi}{Definition}

\newtheorem{rem}[thm]{Remark}
\numberwithin{equation}{section}

\newcommand{\Z}{{\mathbb Z}}
\newcommand{\R}{{\mathbb R}}
\newcommand{\C}{{\mathbb C}}

\newcommand{\N}{{\mathbb N}}

\newcommand{\al}{{\alpha}}
\newcommand{\be}{{\beta}}
\newcommand{\de}{{\delta}}
\newcommand{\eps}{{\varepsilon}}
\newcommand{\gm}{{\gamma}}

\newcommand{\Om}{{\Omega}}

\newcommand{\ka}{{\kappa}}
\newcommand{\si}{{\sigma}}
\newcommand{\lm}{{\lambda}}
\newcommand{\ph}{{\varphi}}

\newcommand{\ov}[1]{\overline{ #1}}

\newcommand{\Hm}[1]{\leavevmode{\marginpar{\tiny%
$\hbox to 0mm{\hspace*{-0.5mm}$\leftarrow$\hss}%
\vcenter{\vrule depth 0.1mm height 0.1mm width \the\marginparwidth}%
\hbox to 0mm{\hss$\rightarrow$\hspace*{-0.5mm}}$\\\relax\raggedright
#1}}}

\begin{document}
\title[Eigenvalue asymptotics for Schr\"odinger operators on sparse graphs]{Eigenvalue asymptotics  for Schr\"odinger operators on sparse graphs}

\author[M. Bonnefont]{Michel Bonnefont} \address{Michel Bonnefont, Institut de   Math\'ematiques de Bordeaux Universit\'e Bordeaux
351, cours de la Lib\'eration
F-33405 Talence cedex, France} \email{michel.bonnefont@math.u-bordeaux.fr}

\author[S. Gol\'enia]{Sylvain Gol\'enia} \address{Sylvain Gol\'enia, Institut de
  Math\'ematiques de Bordeaux Universit\'e Bordeaux
351, cours de la Lib\'eration
F-33405 Talence cedex, France} \email{sylvain.golenia@math.u-bordeaux.fr}

\author[M. Keller]{Matthias Keller}
\address{Matthias Keller, Friedrich Schiller Universit\"at Jena, Mathematisches Institut, 07745 Jena, Germany} \email{m.keller@uni-jena.de}


\date{\today}

\begin{abstract} \noindent
We consider Schr\"odinger operators on sparse graphs. The geometric definition of sparseness turn out to be equivalent to a functional inequality for the Laplacian. In consequence, sparseness has in turn strong spectral and functional analytic consequences. Specifically, one consequence is that it allows to completely describe the form domain. Moreover, as another consequence it leads to a characterization for discreteness of the spectrum. In this case we determine the first order of the corresponding eigenvalue asymptotics.
\end{abstract}

\maketitle

\section{Introduction}
The spectral theory of discrete Laplacians on finite or infinite graphs has drawn a lot of attention for decades. One important aspect  is to understand the relations between the geometry of the graph and the spectrum of the Laplacian.
 Often a particular focus lies on the study of the bottom of the spectrum and the eigenvalues below the essential spectrum.

Certainly the most well-known estimates for the bottom of the spectrum of Laplacians on infinite graphs are so called isoperimetric estimates or Cheeger inequalities. Starting with \cite{D1} in the case of infinite graphs, these inequalities were intensively studied and resulted in huge body of literature, where we here mention only \cite{BHJ, BKW,  D2, DK,  Fu,M1,M2, K1, KL2,Woj1}. In certain more specific geometric situations the bottom of the spectrum might be estimated in terms of curvature, see \cite{BJL,H,JL, K1,K2,KP, LY, Woe}.  There are
various other more recent approaches such as Hardy inequalities in \cite{Gol2} and summability criteria involving the boundary and volume of balls  in \cite{KLW}.

In this work we focus on sparse graphs to study discreteness of spectrum and eigenvalue asymptotics.   In a moral sense, the term sparse means  that there are not `too many' edges, however, throughout the years various different definitions were investigated. We mention here       \cite{EGS,Lo} as seminal works which are closely related to our definitions. As it is impossible to give a   complete discussion of the development, we refer to some selected more recent works such as    \cite{AABL,B,LS,M2} and references therein which also illustrates the great variety of possible definitions. Here, we discuss three notions of sparseness that result in a hierarchy of very general classes of graphs.

Let us highlight the work of Mohar \cite{M3}, where large eigenvalues of the adjacency
matrix on finite graphs are studied. Although our situation of
infinite graphs with unbounded geometry requires fundamentally
different techniques -- functional analytic
rather than combinatorial --  in spirit our work is certainly closely
related.

The techniques used in this paper owe on the one hand to considerations of  isoperimetric estimates as well as a scheme developed in \cite{Gol2} for the special case of trees.
In particular, we show that a notion of sparseness is  a geometric characterization for an inequality of the type
\begin{align*}
    (1-a)\deg - k\leq \Delta \leq (1+a)\deg +k
\end{align*}
for some $a\in(0,1)$, $k\ge0$ which holds in the form sense (precise definitions and details will be given below). The moral of this inequality is that the asymptotic behavior of the Laplacian $\Delta$ is controlled by the vertex degree function $\deg$ (the smaller $a$  the better the control).

Furthermore, such an inequality has very strong consequences which  follow from well-known functional analytic principles. These consequences include an explicit description of the form domain, characterization for discreteness of spectrum and eigenvalue asymptotics.

Let us set up the framework. Here, a graph $\Gr$ is a pair   $(\Vr,\Er)$,
where $\Vr$ denotes a countable set of  vertices and
$\Er:\Vr\times \Vr \to \{0,1\}$ is a symmetric function with zero diagonal determining the
edges.  We say two vertices $x$ and $y$ are \emph{adjacent} or \emph{neighbors}  whenever $\Er(x,y)=\Er(y,x)=1$. In this case,  we write $x\sim  y$ and we call $(x,y)$ and $(x,y)$ the \emph{(directed) edges} connecting $x$ and $y$. We assume that
$\Gr$ is \emph{locally finite} that is  each vertex has
only   finitely many neighbors.  For any finite set
$\Wr\subseteq \Vr$, the \emph{induced subgraph}
$\Gr_{\Wr}:=(\Wr,\Er_{\Wr})$ is defined by setting
$\Er_{\Wr}:=\Er|_{\Wr\times \Wr}$, i.e., an edge is contained in  $\Gr_{\Wr}$ if and only if both of its
vertices are in $\Wr$.

We consider the complex Hilbert space
$\ell^{2}(\Vr):=\{\ph:\Vr\to\C$ such  that $ \sum_{x\in
  \Vr}|\ph(x)|^{2}<\infty\}$ endowed with the scalar product
$\langle{\ph,\psi}\rangle:=\sum_{x\in \Vr}\ov{\ph(x)}\psi(x)$, $\ph,\psi\in\ell^{2}(\Vr)$.

For a function $g:\Vr\to\C$, we denote the operator of multiplication by $g$ on $\ell^{2}(\Vr)$ given by $\ph\mapsto g\ph$ and domain $\Dc(g):=\{\ph\in\ell^{2}(\Vr)\mid g\ph\in \ell^{2}(\Vr)\}$ with slight abuse of notation also by $g$.

Let $q:\Vr\to [0,\infty)$. We consider the Schr\"odinger operator $\Delta+q$ defined as
\begin{align*}
\Dc(\Delta+q):=\Big\{\ph\in\ell^{2}(\Vr)\mid& \Big(v\mapsto \sum_{w\sim
  v}(\ph(v)-\ph(w)) + q(v)\varphi(v)\Big)\in\ell^{2}(\Vr)\Big\}\\
    (\Delta+q)\ph(v)&:=\sum_{w\sim v}(\ph(v)-\ph(w)) + q(v) \varphi(v).
\end{align*}
The operator is non-negative and selfadjoint as it is
essentially selfadjoint on $\Cc_c(\Vr)$, the set of finitely
supported functions $\Vr\to \R$, (confer
\cite[Theorem~1.3.1]{Woj1}, \cite[Theorem~6]{KL1}). In Section~\ref{s:form} we will allow for potentials whose negative part is form bounded with bound strictly less than one. Moreover, in Section~\ref{s:magnetic} we consider also magnetic Schr\"odinger operators.

As mentioned above sparse graphs have already been introduced in various contexts with varying definitions.
In this article we also treat  various natural generalizations of the concept. In this introduction we stick to an intermediate situation.

\begin{defi}
A graph $\Gr:=(\Vr,\Er)$  is called  \emph{$k$-sparse}
 if for any finite set $\Wr\subseteq \Vr$ the induced subgraph
 $\Gr_{\Wr}:=(\Wr,\Er_{\Wr})$ satisfies
\[
2|\Er_{\Wr}| \leq k |\Wr|,
\]
where $|A|$ denotes the cardinality of a finite set $A$ and  we set \[|\Er_{\Wr}|:=\displaystyle\frac{1}{2}|\{(x,y)\in\Wr\times\Wr\mid
\Er_{\Wr}(x,y)=1\}|,\] that is we count the non-oriented
edges in $\Gr_{\Wr}$.
\end{defi}

Examples of sparse graphs are planar graphs and, in particular, trees.
We refer to  Section~\ref{s:sparse} for more examples.

For a function $g:\Vr\to \R$ and a finite set $\Wr\subseteq \Vr$, we
denote
\[g(\Wr):=\sum_{x\in\Wr}g(x).\]
 Moreover,  we define
\begin{align*}
    \liminf_{|x|\to \infty} g(x):=    \sup_{\Wr\subset \Vr\mbox{\scriptsize{ finite}}}\inf_{x\in \Vr\setminus \Wr} g(x), \qquad
    \limsup_{|x|\to \infty} g(x):=    \inf_{\Wr\subset \Vr\mbox{\scriptsize{ finite}}}\sup_{x\in \Vr\setminus \Wr} g(x).
\end{align*}
For two selfadjoint operators $T_{1},T_{2}$ on a Hilbert space and a subspace $\Dc_{0}\subseteq \Dc(T_{1})\cap \Dc(T_{2})$ we write $T_{1}\leq T_{2}$ on $\Dc_{0}$ if $\langle T_{1}\ph,\ph\rangle\leq \langle T_{2}\ph,\ph\rangle$ for all $\ph\in \Dc_{0}$. Moreover, for a selfadjoint semi-bounded operator $T$ on a Hilbert space, we
denote the eigenvalues below the essential spectrum by $\lm_{n}(T)$, $n\ge0$,  with increasing order counted with multiplicity.

The next theorem is a special case of the more general Theorem
\ref{t:form} in Section~\ref{s:form}. It  illustrates our results in
the case of sparse graphs introduced above and includes the case of
trees, \cite[Theorem~1.1]{Gol2}, as a special case. While the proof in
\cite{Gol2} uses a
Hardy inequality, we rely on some new ideas which have their roots in
isoperimetric techniques. The proof is given in Section~\ref{t:form}.

\begin{thm}\label{t:sparse} Let $\Gr:=(\Vr, \Er)$ be a $k$-sparse graph and
  $q:\Vr\to [0, \infty)$. Then, we have the following:
  \begin{itemize}
    \item [(a)] For all $0<\eps\le 1$,
 \begin{equation*}\label{e:main-bound-sparse}
  ({1-\eps})(\deg+q)-\frac{k}{2}\left(\frac{1}{\eps}- \eps\right) \le  \Delta+q\le({1+\eps})
  (\deg+q) +\frac{k}{2}\left(\frac{1}{\eps}- \eps\right),
 \end{equation*}
on $\Cc_c(\Vr)$.
    \item[(b)] $\Dc\left((\Delta+q)^{1/2}\right)=
\Dc\left((\mathrm{deg}+q)^{1/2}\right)$.
    \item[(c)] The operator $\Delta+q$ has
purely discrete spectrum if and only if $$\liminf_{|x|\to\infty}(\mathrm{deg}+q)(x)=\infty.$$ In this case, we obtain
\begin{align*}
    \liminf_{\lm\to\infty} \frac{\lm_{n}(\Delta+q)} {{\lm}_{n}(\mathrm{deg}+q)}=1.
\end{align*}
  \end{itemize}
\end{thm}

As a corollary, we obtain following  estimate for the bottom and the top of the (essential) spectrum.

\begin{cor} \label{c:sparse} Let $\Gr:=(\Vr, \Er)$ be a $k$-sparse graph and $q:\Vr\to [0, \infty)$. Define $d:= \inf_{x\in \Vr} (\deg+q)(x)$ and
 $D:= \sup_{x\in \Vr}(\deg+q)(x)$. Assume $d<k\leq D<+\infty$, then
 \[  d  -2 \sqrt{\frac{k}{2} \left(d- \frac{k}{2}\right)} \leq \inf \sigma(\Delta+q)\leq \sup \sigma(\Delta+q) \leq D -2 \sqrt{\frac{k}{2} \left(D- \frac{k}{2}\right)}.
\]
 Define
 $d_{\rm ess}:=\liminf_{|x|\to\infty} (\deg+q)(x)$ and $D_{\rm ess}:=\limsup_{|x|\to\infty} (\deg+q)(x)$. Assume $d_{\rm ess}<k\leq D_{\rm ess}<+\infty$, then
\[
d_{\rm ess} -2 \sqrt{\frac{k}{2} \left(d_{\rm ess}- \frac{k}{2}\right)}
\le
\inf \sigma_{\rm ess}(\Delta+q) \le
\sup \sigma_{\rm ess}(\Delta+q) \leq D_{\rm ess} -2 \sqrt{\frac{k}{2} \left(D_{\rm ess}- \frac{k}{2}\right)}.
\]
\end{cor}

\begin{proof}[Proof of Corollary \ref{c:sparse}]
The conclusion  follows by taking $\eps=\min\Big(\sqrt \frac{k}{2d-k},1\Big)$ in (a) in of Theorem~\ref{t:sparse}.
\end{proof}

\begin{rem} The bounds in Corollary~\ref{c:sparse}    are optimal for
  the bottom and the top of the (essential) spectrum in the case of
  regular trees. \end{rem}

The paper is structured as follows. In the next section an extension
of the notion of sparseness is introduced which is shown to be
equivalent to a functional inequality and equality of the form domains
of $\Delta$ and $\deg$. In Section~\ref{s:almost-sparse} we consider
almost sparse graphs for which we obtain precise eigenvalue
asymptotics. Furthermore, in Section~\ref{s:magnetic} we shortly
discuss magnetic Schr\"odinger operators. Our notion of sparseness has
very explicit but non-trivial connections to isoperimetric
inequalities which are made precise in
Section~\ref{s:cheeger}. Finally, in Section~\ref{s:sparse} we discuss
some examples.

\section{A geometric characterization of the form domain}\label{s:form}

In this section we  characterize equality of the form domains of
$\Delta+q$ and $\deg+q$ by a geometric property. This geometric
property is a generalization of the notion of sparseness from the
introduction. Before we come to this definition, we  introduce the
class of potentials that is treated in this paper.

Let $\al>0$. We say a potential $q:\Vr\to\R$  is in the class
$\Kr_{\al}$ if there is $C_\al \ge0$ such that
\begin{align*}
    q_{-}\leq \al(\Delta+q_{+})+C_\al,
\end{align*}
where  $q_{\pm}:=\max(\pm q,0)$. For $\al\in(0,1)$, we define the
operator $\Delta+q$ via the form sum of the operators $\Delta+q_{+}$
and $-q_{-}$ (i.e., by the KLMN Theorem, see {e.g.,}
\cite[Theorem~X.17]{RS}). Note that $\Delta+q$ is bounded from below
and
$$\Dc(|\Delta+q|^{\frac{1}{2}}) =\Dc((\Delta+q_{+})^{\frac{1}{2}})
=\Dc(\Delta^{\frac{1}{2}})\cap\Dc(q_{+}^{\frac{1}{2}}),$$
where $|\Delta+q|$ is defined by the spectral theorem.
The last equality follows from \cite[Theorem~5.6]{GKS}.
in the sense of functions and forms.

An other important class  are the potentials
\begin{align*}
    \Kr_{{0^+}}:=\bigcap_{\al\in(0,1)}\Kr_{\al}.
\end{align*}
In our context of sparseness, we can  characterize the class
$\Kr_{0^+}$ to be the potentials whose negative part $q_{-}$ is morally
 $o(\deg+q_{+})$, see Corollary~\ref{c:potentials}.
Let us mention that if $q_{-}$ is in the Kato class with respect to
$\Delta+q_{+}$, i.e., if we have
$\limsup_{t\to0 {^+}}\|e^{-t(\Delta+q_{+})}q_{-}\|_{\infty}=0$, then $q:=q_+-q_-\in
\Kr_{0^+}$ by \cite[Theorem~3.1]{SV}.

Next, we come to an extension of the notion of sparseness.
For a set $\Wr\subseteq \Vr$, let the boundary $\partial\Wr$ of $\Wr$ be the set of edges emanating from $\Wr$
\begin{align*}
    \partial\Wr :=\{(x,y)\in \Wr\times \Vr\setminus\Wr\mid x\sim y\}.
\end{align*}

\begin{defi}
Let $\Gr:=(\Vr, \Er)$ be a graph and $q:\Vr\to \R$.
For given
$a \geq 0$ and $k\geq 0$, we say that
$(\Gr, q)$ is \emph{$(a,k)$-sparse} if for any finite set
$\Wr\subseteq \Vr$ the induced subgraph $\Gr_{\Wr}:=(\Wr,\Er_{\Wr})$ satisfies
\[
2|\Er_{\Wr}| \leq k |\Wr| + a( |\partial \Wr|+q_{+}(\Wr)).
\]
\end{defi}

\begin{rem}
(a) Observe that the definition depends only on $q_{+}$.
{The negative part of $q$ will be taken in account through the
hypothesis $\Kr_{\al}$ or $\Kr_{0^+}$  in our theorems.}\\
(b) If $(\Gr,q)$ is $(a,k)$-sparse, then $(\Gr,q')$ is $(a,k)$-sparse
for every $q'\ge q$.
\\
(c) As mentioned above there is a great variety of
  definitions  which were  so far  predominantly established for
  (families of) finite graphs. For example  it is asked that
  $|\Er|=C|\Vr|$  in \cite{EGS},  $|\Er_{\Wr}|\leq k|\Wr|+l$ in
  \cite{Lo,LS}, $|\Er|\in O(|\Vr|)$ in \cite{AABL} and $\deg(\Wr)\leq
  k|\Wr|$ in \cite{M3}.
\end{rem}

We now characterize  the equality of the form domains in
geometric terms.

\begin{thm}\label{t:form}
Let $\Gr:=(\Vr, \Er)$ be a graph and $q\in \Kr_{\al}$,
$\al\in(0,1)$. The following  assertions are equivalent:
\begin{itemize}
  \item [(i)]  There are $a,k\geq 0$ such that $(\Gr, q)$ is $(a,k)$-sparse.
  \item [(ii)] There are $\tilde a\in (0,1)$ and $\tilde k\geq 0$ such that on $\Cc_c(\Vr)$
\begin{align*}\label{e:formath}
(1-\tilde a) (\deg+q) - \tilde k \leq \Delta + q \leq (1+\tilde a)
(\deg+q) +\tilde k.
\end{align*}
  \item [(iii)] There are $\tilde a\in (0,1)$ and $\tilde k\geq 0$ such that on $\Cc_c(\Vr)$
\begin{align*}
(1-\tilde a) (\deg+q) - \tilde k \leq \Delta + q. \end{align*}
  \item [(iv)] $\Dc(|\Delta + q|^{1/2})= \Dc(|\deg
    +q|^{1/2})$.
\end{itemize}
Furthermore,  $\Delta+q$ has
purely discrete spectrum if and only if
\[\liminf_{|x|\to\infty}(\mathrm{deg}+q)(x)=\infty.\]
 In this case, we obtain
\begin{align*}
1-\tilde a \leq \liminf_{n\to \infty}
\frac{\lambda_n(\Delta+q)}{\lambda_n(\deg +q)}\leq \limsup_{n\to \infty}
\frac{\lambda_n(\Delta+q)}{\lambda_n(\deg +q)}\leq  1+\tilde a.
\end{align*}
\end{thm}

Before we come to the proof of Theorem~\ref{t:form}, we summarize the
relation between the sparseness parameters $(a,k)$ and the constants
$(\tilde a,\tilde k)$ in the inequality in Theorem~\ref{t:form}~(ii).


\begin{rem}\label{r:link}
Roughly speaking $a$ tends to $\infty$ as $\tilde a$ tends to $1^-$
and $a$ tends to $0^+$ as $\tilde a$ tends to $0^+$ and
vice-versa. More precisely, Lemma~\ref{l:inequality_implies_sparse}
we obtain that for given $\tilde a$ and $\tilde k$ the values of $a$
and $k$ can be chosen to be
\begin{align*}
    a=\frac{\tilde a}{1-\tilde a}\quad\mbox{and}\quad k=\frac{
      \tilde k}{1-\tilde a}.
\end{align*}
Reciprocally, given $a,k\geq 0$ and $q:\Vr\to[0,\infty)$,
Lemma~\ref{l:sparse_implies_inequality} distinguishes the case where
the graph is sparse $a=0$ and $a>0$. For $a=0$ we  may
 choose $\tilde a\in(0,1)$ arbitrary and
   $$\tilde k=\frac{k}{2}\Big(\frac{1}{\tilde a}-\tilde a\Big).$$
For an $(a,k)$-sparse graph with $a>0$ the precise constants are found
below in Lemma~\ref{l:sparse_implies_inequality}. Here, we discuss the
asymptotics. For $a\to 0^+$, we obtain
\[\tilde a \simeq \sqrt{2a}\quad \mbox{ and
}\quad \tilde k \simeq \frac{k}{2a}, \]
 and for $a\to \infty$
\[  \tilde a \simeq 1- \frac{3}{8 a^2}\quad \mbox{ and } \quad \tilde k \simeq
\frac{3k}{4a}.\] In the case $q\in \Kr_\al$, the constants $\tilde
a$, $\tilde k$ from the case $q\ge0$ have to  be replaced by
constants  whose formula can be explicitly read from
Lemma~\ref{l:formperturbation}. For $\al\to 0^+$, the constant
replacing $\tilde a$ tends to $\tilde a$ while the asymptotics of
the constant replacing $\tilde k$ depend also on the behavior of
$C_{\al}$ from the assumption $q_{-}\leq\al(\Delta+q)+C_{\al}$.
\end{rem}

\begin{rem}\label{r:form}
(a) Observe that in the context of
  Theorem~\ref{t:form} statement (iv) is equivalent to
\begin{itemize}
  \item [(iv')] $\Dc(|\Delta + q|^{1/2})= \Dc((\deg +q_{+})^{1/2}).$
\end{itemize}
Indeed, (ii) implies the corresponding inequality for $q=q_{+}$. Thus,
as $q\in\Kr_{\al}$,
\begin{align*}
  \Dc(|\Delta+q|^{\frac{1}{2}})=
        \Dc((\Delta+q_{+})^{\frac{1}{2}})=\Dc((\deg+q_+)^{\frac{1}{2}}).
\end{align*}
\\
(b) The definition of the class $\Kr_{0^+}$ is rather abstract. Indeed,
Theorem~\ref{t:form} yields a very concrete characterization of these
potentials, see Corollary~\ref{c:potentials} below.
\\
(c)  Theorem~\ref{t:form} characterizes equality of the form
domains. Another natural question is under which circumstances the
operator domains agree. For a discussion on this matter we refer to
\cite[Section~4.1]{Gol2}.
\end{rem}

The rest of this section is devoted to the proof of the results which
are divided into three parts. The following three lemmas essentially
show the equivalences  (i)$\Leftrightarrow$(ii)$\Leftrightarrow$(iii)
providing the explicit dependence of $(a,k)$ on $(\tilde a,\tilde k)$
and vice versa. The third part uses general functional analytic
principles collected in the appendix.

The first lemma shows (iii)$\Rightarrow$(i).

\begin{lemma}\label{l:inequality_implies_sparse} Let $\Gr:=(\Vr, \Er)$
  be a graph and $q: \Vr \to \R$.
 If  there are $\tilde a\in (0,1)$ and $\tilde k\geq 0$ such that
for all  $f$ in $\Cc_c(\Vr)$,
\begin{align*}
(1-\tilde a) \langle f,(\deg +q) f \rangle - \tilde k \Vert f\Vert^2
\leq \langle f,{\Delta f +qf} \rangle,
\end{align*}
then $(\Gr,q)$ is $(a,k)$-sparse with
\begin{align*}
    a=\frac{\tilde a}{1-\tilde a}\quad\mbox{and}\quad k=\frac{
      \tilde k}{1-\tilde a}.
\end{align*}
\end{lemma}
\begin{rem}
{We stress that we suppose solely that $q:\Vr \to
  \R$ and work with $\Delta|_{\Cc_c(\Vr)}+ q|_{\Cc_c(\Vr)}$.
We do not specify any self-adjoint extension of the latter.}
\end{rem}

\begin{proof} {Let $f \in \Cc_c(\Vr)$.}
By adding $q_{-}$ to the assumed inequality we obtain immediately
 \begin{align*}
    (1-\tilde a) \langle f,(\deg +q_{+}) f \rangle -\tilde k \Vert f\Vert^2  \leq \langle f,\Delta f +q_{+} f \rangle.
 \end{align*}
Let $\Wr\subseteq \Vr$ be a finite set and denote by $\bone_\Wr$ the
characteristic function of the set $\Wr$. We recall the basic
equalities
\begin{align*}
\deg(\Wr)=2|\Er_{\Wr}|+|\partial\Wr| \quad\mbox{and}\quad
\langle{\bone_{\Wr},\Delta\bone_{\Wr}}\rangle =|\partial W|.
\end{align*}
Therefore, applying the asserted inequality {with} $f=\bone_\Wr$, we
obtain
\begin{align*}
2|\Er_{\Wr}|\leq \frac{\tilde k}{1-\tilde a}|W|+ \frac{\tilde a}{1-\tilde a}\left(|\partial W|
  +q_{+}(\Wr) \right).
 \end{align*}
This proves the statement.
\end{proof}

The second lemma gives (i)$\Rightarrow$(ii) for $q\ge0$.

\begin{lemma}\label{l:sparse_implies_inequality} Let $\Gr:=(\Vr, \Er)$
  be a graph and $q:\Vr\to [0,\infty)$. If there are $a,k\geq 0$ such
  that $(\Gr, q)$ is $(a,k)$-sparse, then
\begin{align*}
(1-\tilde a) (\deg+q) - \tilde k \leq \Delta + q \leq (1+\tilde a)
(\deg+q) +\tilde k.
\end{align*}
on $\Cc_c(\Vr)$,
 where if $(\Gr,q)$ is sparse, i.e., $a=0$, we  may
 choose $\tilde a\in(0,1)$ arbitrary and
  \begin{align*}
    \tilde k=\frac{k}{2}\Big(\frac{1}{\tilde a}-\tilde a\Big).
 \end{align*}
In the other case, i.e. $a>0$, we may choose  \begin{align*}
 \tilde a=\frac{\sqrt{\min\left(\frac{1}{4},
      a^2\right)+2a+a^2}}{(1+a)}\quad\mbox{and}\quad
\tilde k=\max\left(\frac{\max\left(\frac{3}{2},
      \frac{1}{a}-{a}\right)k}{2(1+a)}, 2k(1-\tilde a)\right).
\end{align*}
\end{lemma}
\begin{proof}
Let $f\in \cC_{c}(\Vr)$ be complex valued.
Assume first that $\langle f,(\deg+q)f\rangle<k\|f\|^{2}$. In this
case, remembering $\Delta\leq 2\deg$, we can choose $\tilde a\in(0,1)$
arbitrary and $\tilde k$ such that
\begin{align*}
    \tilde k\ge 2(1-\tilde a)k.
\end{align*}
So, assume $\langle f,(\deg+q)f\rangle\ge k\|f\|^{2}$.
Using  an area and a co-area formula (cf.\ \cite[Theorem~12 and
Theorem~13]{KL2}) with
\begin{align*}
\Om_{t}:=\{x\in \Vr\mid |f(x)|^{2}>t\},
\end{align*}
in the first step and
   the assumption of sparseness in the third step, we obtain
\begin{align*}
 \langle f,&(\deg +q) f\rangle-
k\|f\|^{2}=\int_{0}^{\infty}\Big(\deg(\Om_{t}) + q(\Om_t)-k
|\Om_{t}|\Big)dt
\\
&=\int_{0}^{\infty}\Big(2|\Er_{\Omega_{t}}|+|\partial\Om_{t}| + q(\Om_t)-k
|\Om_{t}|\Big)dt
\\
&\leq (1+a)\int_{0}^{\infty}|\partial\Om_{t}| +q(\Om_t) dt \\
&=\frac{(1+a )}{2}\sum_{x,y,x\sim
  y}\left| |f(x)|^{2}-|f(y)|^{2}\right| + (1+a )\sum_x
q(x) |f(x)|^2
\\
&\leq\frac{(1+a )}{2}\sum_{x,y,x\sim
  y}|(f(x)-f(y))(\overline {f(x)}+ \overline {f(y)})|+
  (1+a )\sum_x q(x) |f(x)|^2\\
&\leq \frac{(1+a )}{2} \left(\sum_{x,y,x\sim y} |f(x)-f(y)|^2
+   2\sum_x q(x)|f(x)|^2 \right)^{1/2}
\\
&\quad \quad\quad\quad\quad\quad\quad\times
\left(\sum_{x,y,x\sim y} |f(x)+f(y)|^2 +
  2\sum_x q(x)|f(x)|^2 \right)^{1/2}
\\\displaybreak
&=  (1+a) \langle f,(\Delta+q) f\rangle^{\frac{1}{2}}\big(2\langle
f,(\deg+q) f\rangle- \langle f,(\Delta+q)
f\rangle\big)^{\frac{1}{2}},
\end{align*}
where we used the Cauchy-Schwarz inequality in the last inequality and
basic algebraic manipulation in the last equality. Since the left hand
side is non-negative by the assumption $\langle f,(\deg+q)f\rangle\ge
k\|f\|^{2}$, we can take square roots on both sides. To shorten
notation, we assume for the rest of the proof $q\equiv 0$ since the
proof with $q\neq 0$ is completely analogous.

Reordering the terms, yields
\begin{align*}
 (1+a)^2 \langle f,\Delta f\rangle^{2}-2
 (1+a)^2\langle f,\deg  f\rangle \langle f,\Delta f\rangle+(\langle f,(\deg-k) f\rangle)^2\le0.
\end{align*}
Resolving the quadratic expression above gives,
\begin{align*}
 \langle f, \deg f \rangle -\sqrt{\de}\leq\langle f,\Delta
 f\rangle\leq   \langle f,\deg f \rangle +\sqrt{\de},
\end{align*}
with $$\delta:= \langle f, \deg f \rangle^2 - (1+a)^{-2} (\langle f,(\deg-k)  f \rangle)^{2}.$$

Using $4\xi\zeta\leq (\xi+\zeta)^2$, $\xi,\zeta\ge0$,  for all $0<\lambda <1$, we  estimate $\de$  as follows
\begin{align*}
(1+a)^{2}\de &=    (2a+a^2)  \langle f, \deg f \rangle^2 +
  k \|f\|^{2}\langle f, (2\deg-k) f \rangle
 \\
\nonumber
&\leq  (2a+a^2)  \langle f, \deg f \rangle^2  + \left( \lambda  \langle f, \deg f \rangle + \frac{k}{2} \left(\frac{1}{\lambda}- \lambda \right)\|f\|^2\right)^2\\
&\leq  \left(  \sqrt{ \lambda^ 2+2a+a^2}   \langle f, \deg f \rangle + \frac{k}{2} \left(\frac{1}{\lambda}- \lambda \right)\|f\|^2\right)^2.
\end{align*}

If $a=0$, i.e., the $k$-sparse case, then we take $\lambda=\tilde a$ to get
\begin{align*}
\de &\leq    k \|f\|^{2}\langle f, 2\deg f \rangle \leq
\Big(\tilde a \langle f, \deg f\rangle + \frac{k}{2} \left(\frac{1}{ \tilde a}- \tilde a \right)
    \|f\|^2   \Big)^2.
\end{align*}
As $k/ 2 \tilde a\ge 2(1-\tilde a)k$, this proves the desired inequality with $\tilde k = k/ 2 \tilde a$.

If $a>0$, we  take $\lambda= \min\left(\frac{1}{2},a\right)$ to get
\begin{align*}
(1+a)^{2}\de &=  \left(  \left(\sqrt{\min\left(\frac{1}{4}, a^2 \right)+2a+a^2} \right)  \langle f, \deg f \rangle + \frac{k}{2} \max\left(\frac{3}{2},\left(\frac{1}{a}- a\right) \right)\|f\|^2\right)^2.
\end{align*}
Keeping in mind the restriction  $\tilde k\ge2(1-\tilde a)k$ for the case $\langle f,(\deg+q)f\rangle< k\|f\|^{2}$, this gives the statement with the choice of $(\tilde a,\tilde k)$ in the statement of the lemma.
\end{proof}

The two lemmas above are sufficient to prove Theorem~\ref{t:form} for
the case $q\ge0$. An application of Lemma~\ref{l:formperturbation}
turns the lower bound of Lemma~\ref{l:sparse_implies_inequality} into
a corresponding lower bound. This straightforward argument does not
work for the upper bound. However, the following surprising lemma
shows that such a lower bound by $\deg$ automatically implies the
corresponding upper bound.
There is a deeper reason for this fact which  shows up in the context
of magnetic Schr\"odinger operators. We present the non-magnetic
version of the statement here for the sake of being self-contained in
this section. For the more conceptual and more general magnetic
version, we refer to Lemma~\ref{l:upsidedown}.

\begin{lemma}[Upside-Down-Lemma -- non-magnetic
    version]\label{l:upsidedown1}
 Let $\Gr:=(\Vr, \Er)$   be a graph and $q:\Vr
   \to \R$. Assume  there are $\tilde a\in(0,1),$ $\tilde k\geq
0$ such that for all $f\in \Cc_c(\Vr)$,
\begin{align*}
(1-\tilde a)\langle f, (\deg +q) f \rangle  -\tilde k\Vert
f\Vert^2\leq  \langle f,\Delta f +q f \rangle,
\end{align*}
then  for all $f\in \Cc_c(\Vr)$, we also have
\begin{align*} \nonumber
 \langle f,\Delta f +q f \rangle  \leq  (1+\tilde a)\langle f,
 (\deg +q) f \rangle + \tilde k\Vert f\Vert^2.
\end{align*}
\end{lemma}
\proof By a direct calculation we find for $f\in\Cc_c(\Vr)$
\begin{align*}
    \langle f, (2\deg-\Delta)f\rangle&=\frac{1}{2}\sum_{x,y\in\Vr,x\sim
      y}(2|f(x)|^{2}+ 2|f(y)|^{2})-|f(x)-f(y)|^{2})\\
    &=\frac{1}{2}\sum_{x,y,x\sim y}|f(x)+f(y)|^{2}
    \geq\frac{1}{2}\sum_{x,y,x\sim y}\left| |f(x)|-|f(y)|
    \right|^{2}
\\
&=\langle |f|, \Delta|f|\rangle.
\end{align*}
Adding $q$ to the inequality and using the assumption gives after reordering
\begin{align*}
   \langle f, (\Delta+q)f\rangle-2\langle f, (\deg+q)f\rangle&\leq
   -\langle|f|, (\Delta+q)|f|\rangle\\
   &\leq-(1-\tilde a)\langle |f|, (\deg+q)|f|\rangle +\tilde k\langle
   |f|,|f|\rangle\\
   &=-(1-\tilde a)\langle f,  (\deg+q)f\rangle +\tilde k\langle  f,f\rangle
\end{align*}
which yields the assertion.\qed

\begin{proof}[Proof of Theorem~\ref{t:form}]
The implication (i)$\Rightarrow$(iii) follows from
Lemma~\ref{l:sparse_implies_inequality} applied with $q_+$ and from
Lemma~\ref{l:formperturbation} with $q$. The implication
(iii)$\Rightarrow$(ii) follows from the  Upside-Down-Lemma above.
Furthermore, (ii)$\Rightarrow$(i) is implied by
Lemma~\ref{l:inequality_implies_sparse}. The equivalence
(ii)$\Leftrightarrow$(iv) follows from an application of the Closed
Graph Theorem, Theorem~\ref{t:imt}. Finally, the statements about
discreteness of spectrum and eigenvalue asymptotics follow from an
application of the Min-Max-Principle, Theorem~\ref{t:minmax}.
\end{proof}

\begin{proof}[Proof of Theorem~\ref{t:sparse}] (a) follows from Lemma~\ref{l:sparse_implies_inequality}. The other statements follow directly from Theorem~\ref{t:form}.
\end{proof}

As a corollary we can now determine the potentials in the class $\Kr_{0^+}$ explicitly and give necessary and sufficient
criteria for potentials being in $\Kr_{\al}$, $\al\in(0,1)$.

\begin{cor}\label{c:potentials}Let $(\Gr,q)$ be an   $(a,k)$-sparse
  graph for some $a,k\ge0$.
\begin{itemize}
  \item[(a)] The potential $q$  is in $\Kr_{0^+}$ if and only if  for
    all $\al\in (0,1)$ there is $\ka_\al\ge0$ such that
\begin{align*}
        q_{-}\leq \al (\deg+q_{+})+\ka_\al.
\end{align*}
  \item[(b)] Let $\al\in (0,1)$ and $\tilde a= \sqrt{\min({1/
        4,a^2})+2a+a^{2}}/(1+a)$ (as given by
    Lemma~\ref{l:sparse_implies_inequality}). If there is
    $\ka_{\al}\ge0$ such that
        $q_{-}\leq \al (\deg+q_{+})+\ka_{\al}$, then $q\in
        \Kr_{\al/(1-\tilde a)}$. On the other, hand if $q\in
        \Kr_{\al}$, then there is $\ka_{\al}\ge0$ such that
        $q_{-}\leq \al(1+\tilde a) (\deg+q_{+})+\ka_{\al}$.
\end{itemize}
\end{cor}
\begin{proof}Using the assumption  $q_{-}\leq\al(\deg+q_{-})+\kappa_\al$ and
  the lower bound of Theorem~\ref{t:form}~(ii), we infer
\begin{align*}
q_{-}\leq \al (\deg+q_{+})+ \kappa_\al\leq \frac{\al}{(1-\tilde a)}
(\Delta+q_{+})+ \frac{\al}{(1-\tilde a)} \tilde k+\kappa_\al.
\end{align*}
Conversely, $q\in \Kr_{\al}$ and the upper bound of Theorem~\ref{t:form}~(ii) yields
\begin{align*}
q_{-}\leq \al(\Delta+q_{+}) + C_\al\leq\al{(1+\tilde a)} (\deg+q_{+})+
\al\tilde k+C_\al.
\end{align*}
Hence, (a) follows. For (b), notice that $\tilde a= \sqrt{\min({1/4,a^2})+2a+a^{2}}/(1+a)$ by Lemma~\ref{l:sparse_implies_inequality}.
\end{proof}


\section{Almost-sparseness and asymptotic of eigenvalues}\label{s:almost-sparse}

In this section we prove better estimates on the eigenvalue
asymptotics in a more specific situation. Looking at the inequality in
Theorem~\ref{t:form} (ii)  it seems desirable to have $\tilde a=0$. As
this is impossible when the degree is unbounded, we consider a
sequence of $\tilde a$ that tends to $0$. Keeping in mind
Remark~\ref{r:link}, this leads naturally to the following definition.

\begin{defi} Let  $\Gr:=(\Vr,\Er)$  be a graph and $q:\Vr\to\R$.
We say  $(\Gr, q)$ is  \emph{almost sparse}
if for all $\eps>0$ there is $k_{\eps}\ge0$ such that $(\Gr,q)$ is
$(\eps,k_{\eps})$-sparse, i.e.,   for any finite set $\Wr\subseteq
\Vr$ the induced subgraph $\Gr_{\Wr}:=(\Wr,\Er_{\Wr})$ satisfies
\[
2|\Er_{\Wr}| \leq k_\ve |\Wr| + \ve \left(|\partial \Wr|+q_{+}(\Wr)\right).
\]
\end{defi}

\begin{rem}(a) Every sparse graph $\Gr$ is almost sparse.\\
(b) For an  almost sparse graph  $(\Gr, q)$, every graph $(\Gr, q')$ with $q'\geq q$ is almost sparse.
\end{rem}

The main result of this section shows how the first order of the  eigenvalue asymptotics in the case of discrete spectrum can be determined for almost sparse graphs.

\begin{thm}\label{t:q-sparse-deg}
Let $\Gr:=(\Vr, \Er)$ be a graph and $q\in\Kr_{0^{+}}$. The
following  assertions are equivalent:
\begin{itemize}
  \item [(i)] $(\Gr, q)$ is almost sparse.
  \item [(ii)] For every $\eps>0$ there are $k_{\eps}\geq 0$ such that on $\Cc_c(\Vr)$
\begin{align*}\label{e:formath}
(1-\eps) (\deg+q) - k_{\eps} \leq \Delta + q \leq (1+\eps)
(\deg+q) +k_{\eps}.
\end{align*}
  \item [(iii)] For every $\eps>0$ there are $k_{\eps}\geq 0$ such that on $\Cc_c(\Vr)$
\begin{align*}
(1-\eps) (\deg+q) - k_{\eps} \leq \Delta + q.
\end{align*}
\end{itemize}
Moreover,
 $\Dc((\Delta + q)^{1/2})= \Dc((\deg +q)^{1/2})$ and
 the operator $\Delta+q$ has
purely discrete spectrum if and only if $\liminf_{|x|\to\infty}(\mathrm{deg}+q)(x)=\infty$. In this case, we have
\begin{align*}
\lim_{n\to \infty}
\frac{\lambda_n(\Delta+q)}{\lambda_n(\deg +q)}=1.
\end{align*}
\end{thm}
\begin{proof}
The statement is a direct application of Theorem~\ref{t:form} if one
keeps track of the constants given explicitly by
Lemma~\ref{l:inequality_implies_sparse},
Lemma~\ref{l:sparse_implies_inequality} and
Lemma~\ref{l:formperturbation}.
\end{proof}


\section{Magnetic Laplacians}\label{s:magnetic}
In this section, we consider magnetic Schr\"odinger operators. Clearly, every lower bound can be deduced from Kato's inequality. However, for the eigenvalue asymptotics we also need to prove an upper bound.

We fix a
phase
\[\theta:\Vr\times \Vr\rightarrow \R/2\pi\Z \;\mbox{ such that }\;
\theta(x,y)= - \theta(y,x).\] For a potential $q:\Vr\to[0,\infty)$ we consider the magnetic Schr\"odinger operator
$\Delta_\theta+q$ defined as
\begin{align*}
\Dc(\Delta_\theta+q):=\Big\{\ph\in\ell^{2}(\Vr)\mid& \Big(v\mapsto \sum_{x\sim y}(\ph(x)-e^{\rmi \theta({x,y})}\ph(x)) +
q(x)\varphi(x)\Big)\in\ell^{2}(\Vr)\Big\}\\
    (\Delta_\theta+q)\ph(x)&:=\sum_{x\sim y}(\ph(x)-e^{\rmi\theta({x,y})} \ph(y)) + q(x) \varphi(x).
\end{align*}
A computation for $\ph \in \Cc_c(\Vr)$ gives
\[\langle  \ph, (\Delta_\theta+q)\ph \rangle= \frac{1}{2} \sum_{x,y,
  x\sim y}\left|\ph(x) -e^{\rmi \theta(x,y)}\ph(y)
\right|^2+\sum_{x}q(x)|\ph(x)|^{2}.\]
The operator is non-negative and selfadjoint as it is essentially
selfadjoint on $\Cc_c(\Vr)$ (confer e.g.\ \cite{Gol2}). For $\al>0$,
let $\Kr^{\theta}_{\al}$ be the class of real-valued potentials $q$
such that $q_{-}\leq \al (\Delta_{\theta}+q_{+})+C_{\al}$ for some
$C_{\al}\ge0$. Denote
\[\Kr_{0^{+}}^{\theta}=\bigcap_{\al\in(0,1)}\Kr_{\al}^{\theta}.\]
Again, for $\al\in(0,1)$ and $q\in\Kr^{\theta}_{\al}$, we define
$\Delta_{\theta}+q$ to be the form sum of $\Delta_{\theta}+q_{+}$ and
$-q_{-}$. 

We present our result for magnetic Schr\"odinger operators
which has one implication from the equivalences of Theorem~\ref{t:form}
and Theorem~\ref{t:q-sparse-deg}.

\begin{thm}\label{t:q-sparse-deg-magne}
Let $\Gr:=(\Vr,\Er)$ be a graph, $\theta$ be a phase and
$q\in \Kr^{\theta}_{0^{+}}$ be a potential.
Assume $(\Gr,q)$ is $(a,k)$-sparse for some  $a,k\geq 0$.
Then, we have the following:
  \begin{itemize}
    \item [(a)] There are $\tilde a\in (0,1)$, $k\ge0$ such that
on $\Cc_c(\Vr)$
 \begin{equation*}\label{e:main-bound-sparse-magnetic}
  ({1-\tilde a})(\deg+q)-k\le  \Delta_{\theta}+q\le({1+\tilde a})
  (\deg+q) +k.
 \end{equation*}
    \item[(b)] $\Dc\left(|\Delta_{\theta}+q|^{1/2}\right)=
\Dc\left(|\mathrm{deg}+q|^{1/2}\right)$.
    \item[(c)] The operator $\Delta_{\theta}+q$ has
purely discrete spectrum if and only if
$$\liminf_{|x|\to\infty}(\mathrm{deg}+q)(x)=\infty.$$ In this case,
if $(\Gr,q)$ is additionally almost sparse, then
\begin{align*}
    \liminf_{\lm\to\infty} \frac{\lm_{n}(\Delta_{\theta}+q)}
    {{\lm}_{n}(\mathrm{deg}+q)}=1.
\end{align*}
  \end{itemize}
\end{thm}
\begin{rem}
(a) The constants $\tilde a$ and $\tilde k$ can chosen to be the
same  as the ones we obtained in the proof of Theorem \ref{t:form},
i.e., these constants are explicitly given combining
Lemma~\ref{l:sparse_implies_inequality} and Lemma~\ref{l:formperturbation}.\\
(b) Statement (a) and (b) of the theorem above remain true  for
$q\in \Kr_{\al}$, $\al\in (0,1)$ since $\Kr_{\al}\subseteq
\Kr_{\al}^{\theta}$ by Kato's inequality below.
\end{rem}

We will prove  the theorem by applying Theorem~\ref{t:form} and
Theorem~\ref{t:q-sparse-deg}.  The considerations heavily rely on
Kato's inequality and  a conceptual version of the
Upside-Down-Lemma, Lemma~\ref{l:upsidedown1}, which shows that a
lower bound for $\Delta+q$ implies an upper and lower bound on
$\Delta_{\theta}+q$. Secondly, in Theorem~\ref{t:q-sparse-deg}
potentials in $\Kr_{0^{+}}$ are considered, while here we start with
the class $\Kr_{0^{+}}^{\theta}$. However, it can be seen that
$\Kr_{0^{+}}=\Kr_{0^+}^{\theta}$ in the  case of $(a,k)$-sparse
graph, see Lemma~\ref{l:classes2} below.

As mentioned above a key fact is Kato's inequality, see e.g.
\cite[Lemma~2.1]{DM} or \cite[Theorem~5.2.b]{GKS}.

\begin{pro}[Kato's inequality]\label{l:Kato}
Let $\Gr:=(\Vr,\Er)$ be a graph, $\theta$ be a phase and  $q:\Vr\to\R$. For all
$f\in \cC_c(\Vc)$, we have
\begin{align*}
  \langle
    |f|,(\Delta|f|+q|f|)\rangle \leq   \langle
    f,(\Delta_{\theta}f+qf)\rangle.
\end{align*}In particular, for all $\al > 0$
\begin{align*}
    \Kr_{\al}\subseteq\Kr_{\al}^{\theta}   \mbox{ and }
      \Kr_{0^+}\subseteq\Kr_{0^+}^{\theta}.
\end{align*}
\end{pro}
\begin{proof}
The proof of the inequality can be obtained by a direct
calculation. The second statement is an immediate consequence.
\end{proof}

The next lemma is a rather surprising observation. It is the magnetic version of the Upside-Down-Lemma, Lemma~\ref{l:upsidedown1}.

\begin{lemma}[Upside-Down-Lemma -- magnetic version]\label{l:upsidedown}
Let $\Gr:=(\Vr,\Er)$ be a graph, $\theta$ be a phase and $q:\Vr\to
\R$ be a potential.  Assume that there are $\tilde
a\in(0,1)$ and $\tilde k\geq 0$ such that
{for all $f\in\Cc_c(\Vr)$, we have
\begin{align*}
(1-\tilde a)\langle f,(\deg + q) f \rangle -\tilde k \Vert
f\Vert^2\leq \langle f, \Delta f +q f\rangle
\end{align*}
then for all $f\in\Cc_c(\Vr)$, we also have
\begin{align*}(1-\tilde a)\langle f,(\deg + q) f \rangle -\tilde k
  \Vert f\Vert^2 &\leq \langle f, \Delta_\theta f +q
  f\rangle\leq(1+\tilde a)\langle f, (\deg + q) f \rangle +\tilde k
  \Vert f\Vert^2.
\end{align*}
}
\end{lemma}
\begin{proof} The lower bound follows directly from Kato's
inequality and the lower bound from the assumption (since $\langle f,(\deg+q)f\rangle=\langle |f|,(\deg+q)|f|\rangle$ for all $f\in \Cc_c(\Vr)$). Now, observe that for all $\theta$
\begin{align*}
    \Delta_{\theta}=2\deg-\Delta_{\theta+\pi}.
\end{align*}
So, the upper bound for $\Delta_{\theta}+q$ follows from the lower bound of $\Delta_{\theta+\pi}+q$ which we deduced from Kato's inequality.
\end{proof}

The lemma above allows to relate the classes $\Kr_{\al}$ and $\Kr_{\al}^{\theta}$ for $(a,k)$-sparse graphs.

\begin{lemma}\label{l:classes2}For $a,k\ge0$ let $\Gr:=(\Vr,\Er)$ be an
  $(a,k)$-sparse graph, $\theta$ be a phase and $\al>0$. Then,
\begin{align*}
    \Kr_{\al}^{\theta}\subseteq \Kr_{\al'}, \mbox{ for  } \al'=
\frac{1+\tilde a}{1-\tilde a}\, \al,
\end{align*}
where $\tilde a$ is given in Lemma~\ref{l:sparse_implies_inequality}. In particular,
\begin{align*}
        \Kr_{0^{+}}^{\theta}= \Kr_{0^{+}}.
\end{align*}
Moreover, if  $(\Gr, q)$ is almost-sparse, then
\[\Kr_\al^\theta \subseteq \Kr_{\al'}, \mbox{ for all } \al'> \al.\]
\end{lemma}
\begin{proof}
Let $q\in \Kr_\al^\theta$. Applying Lemma~\ref{l:sparse_implies_inequality}, we get $\Delta+q_{+}\geq(1-\tilde a)(\deg+q_{+})-\tilde k$. Now, by the virtue of the Upside-Down-Lemma, Lemma~\ref{l:upsidedown}, we infer
\begin{align*}
\Delta_\theta +q_+ &\leq (1+\tilde a)(\deg +q_+) +\tilde k\leq
\frac{1+\tilde a}{1-\tilde a} (\Delta +q_+)+ \frac{2}{1-\tilde a}\tilde k
\end{align*}
which implies the first statement and
$\Kr_{0^{+}}^{\theta}\subseteq \Kr_{0^{+}}$. The reverse inclusion
$\Kr_{0^{+}}^{\theta}\supseteq \Kr_{0^{+}}$ follows from Kato's
inequality, Lemma~\ref{l:Kato}. For almost sparse graphs $a$ can be
chosen arbitrary small and accordingly $\tilde a$ (from
Lemma~\ref{l:sparse_implies_inequality}) becomes arbitrary
small. Hence, the statement $\Kr_\al^\theta \subseteq \Kr_{\al'},$ for
$\al'> \al$ follows from the inequality above.
\end{proof}

\begin{proof}[Proof of Theorem~\ref{t:q-sparse-deg-magne}] Let
$q\in \Kr_{0^{+}}^{\theta}$. By Lemma~\ref{l:classes2}, $q\in
\Kr_{0^{+}}$. Thus, (a) follows from
  Theorem~\ref{t:form} and Lemma~\ref{l:upsidedown}. Using (a)
  statement (b) follows from an application of the Closed Graph
  Theorem, Theorem~\ref{t:imt} and statement (c) follows from an
  application of the Min Max Principle, Theorem~\ref{t:minmax}.
\end{proof}

\begin{rem}
Instead of using Kato's inequality one can also reproduce the proof of
Lemma~\ref{l:sparse_implies_inequality} using the following estimate
\begin{align*}
    ||f(x)|^{2}-|f(y)|^{2}|\leq |(f(x)-e^{\rmi \theta(x,y)}f(y))(\overline{f(x)}+e^{-\rmi \theta(x,y)}\overline {f(y)})|.
\end{align*}
So, we infer the key estimate:
\begin{align*}
 \langle f,&(\deg +q) f\rangle-
k\|f\|^{2}
\\
&\leq  (1+a) \langle f,(\Delta_\theta+q)
f\rangle^{\frac{1}{2}}\langle f,(\Delta_{\theta+\pi}+q)
f\rangle^{\frac{1}{2}},
\\
&=
(1+a) \langle f,(\Delta_\theta+q) f\rangle^{\frac{1}{2}}\big(2\langle
f,(\deg+q) f\rangle- \langle f,(\Delta_\theta+q) f\rangle\big)^{\frac{1}{2}}.
\end{align*}
The rest of the proof is analogous.
\end{rem}

It can be observed that unlike in Theorem~\ref{t:form} or Theorem~\ref{t:q-sparse-deg} we do not have an equivalence in the theorem above. A reason for this seems to be  that our definition of sparseness does not involve the magnetic potential. This direction shall be pursued in the future. Here, we restrict ourselves to some  remarks on the perturbation theory in the context of Theorem~\ref{t:q-sparse-deg-magne} above.

\begin{rem}\label{r:perturbation}(a) If the inequality Theorem~\ref{t:q-sparse-deg-magne}~(a) holds for some $\theta$, then the inequality holds with the same constants for $-\theta$ and $\theta\pm\pi$. This can be seen by the fact $\Delta_{\theta+\pi}=2\deg-\Delta$ and $\langle f,\Delta_{\theta}f\rangle=\langle \overline{f},\Delta_{-\theta}\overline{f}\rangle$ while $\langle f,\deg f\rangle=\langle \overline{f},\deg\overline{f}\rangle$ for $f\in \cC_{c}(\Vr)$.

(b) The set of $\theta$ such that
Theorem~\ref{t:q-sparse-deg-magne}~(a)
holds true for some fixed $\tilde a$ and $\tilde k$ is closed in the product
topology, i.e., with respect to  pointwise convergence. This follows as $\langle f,\Delta_{\theta_{n}}f\rangle\to \langle f,\Delta_{\theta}f\rangle$ if $\theta_{n}\to\theta$, $n\to\infty$, for fixed $f\in \cC_{c}(\Vr)$.

(c) For two phases $\theta$ and $\theta'$ let $h(x)=\max_{y\sim x}| \theta(x,y)-\theta'(x,y)|$. By a straight forward estimate $\limsup_{|x|\to\infty}h(x)=0$ implies that for every $\eps>0$ there is $C\ge0$ such that
\begin{align*}
    -\eps \deg -C\leq \Delta_{\theta}-\Delta_{\theta'}\leq \eps\deg+C
\end{align*}
on $\cC_{c}(\Vr)$. We discuss three consequences of this inequality:

First of all, this inequality immediately yields that if   $\Dc(\Delta_{\theta}^{1/2})=
\Dc(\mathrm{deg}^{1/2})$ then $\Dc(\Delta_{\theta'}^{1/2})=
\Dc(\mathrm{deg}^{1/2})$ (by the KLMN Theorem, see {e.g.,}
\cite[Theorem~X.17]{RS}) which in turn yields
equality of the form domains of $\Delta_{\theta}$ and $\Delta_{\theta'}$.

Secondly, combining this inequality with Theorem~\ref{t:q-sparse-deg}
we obtain the following: If  $\limsup_{|x|\to\infty}\max_{y\sim x}|\theta(x,y)|=0$
and for every $\eps>0$ there is $k_{\eps}\ge0$  such that
\begin{align*}
    (1-\eps)\deg-k_{\eps}\leq\Delta_{\theta}\leq     (1+\eps)\deg+k_{\eps}^,
\end{align*}
then  the graph is almost sparse and in consequence the inequality in Theorem~\ref{t:q-sparse-deg-magne}~(a) holds for any phase.

Thirdly, using the techniques in the proof of \cite[Proposition~5.2]{Gol2} one shows that the essential spectra of $\Delta_{\theta}$ and $\Delta_{\theta'}$ coincide. With slightly more effort and the help of the Kuroda-Birman Theorem, \cite[Theorem~XI.9]{RS} one can  show that if  $h\in \ell^{1}(\Vr)$, then even the absolutely continuous spectra of $\Delta_{\theta}$ and $\Delta_{\theta'}$  coincide.
\end{rem}

\section{Isoperimetric estimates  and sparseness}\label{s:cheeger}
In this section we relate the concept of sparseness with the  concept of isoperimetric estimates.
First, we present a result which should be viewed in the light of Theorem~\ref{t:form} as it points out in which sense isoperimetric estimates are stronger than our notions of sparseness. In the second subsection, we present a result related to Theorem~\ref{t:q-sparse-deg}. Finally, we present a concrete comparison of sparseness and isoperimetric estimates. As this section is of a more geometric flavor we restrict ourselves to the case of potentials $q:\Vr\to[0,\infty)$.

\subsection{Isoperimetric estimates}
Let  $\Ur\subseteq \Vr$ and define the \emph{Cheeger} or
\emph{isoperimetric constant} of $\Ur$ by
\[\al_{\Ur}:=\inf_{\Wr\subset \Ur\,{\scriptsize\mbox{finite}}}
\frac{|\partial \Wr|+q(\Wr)}{\deg(\Wr)+q(\Wr)}.\]
In the case where $\deg(\Wr)+q(\Wr)=0$, for instance when $W$ is
  an isolated point, by convention the above quotient is set to be equal
  to 0,  Note that $\alpha_{\Ur}\in [0,1) $.

The following theorem illustrates in which sense positivity of the
Cheeger constant is linked with $(a,0)$-sparseness. We refer to
Theorem \ref{t:sparse-isop-q0} for precise constants.

\begin{thm}\label{p:cheeger-form} Given $\Gr:=(\Vr, \Er)$ a graph and
  $q:\Vr\to [0,\infty)$. The following assertions are
  equivalent
\begin{itemize}
  \item [(i)] $\alpha_{\Vr}>0.$
  \item [(ii)]
  There is $ \tilde a \in(0,1)$
\begin{align*}
(1- \tilde a) (\deg+q) \leq \Delta + q \leq (1+ \tilde a)(\deg+q).
\end{align*}
  \item [(iii)] There is $\tilde a\in(0,1)$ such that
  \begin{align*}
    (1-\tilde a) (\deg+q) \leq \Delta + q.
  \end{align*}
\end{itemize}
\end{thm}
The implication (iii)$\Rightarrow$(i) is already found in
\cite[Proposition 3.4]{Gol2}. The implication (i)$\Rightarrow$(ii)
is a consequence from standard isoperimetric estimates which can be
extracted from the proof of \cite[Proposition~15]{KL2}.

\begin{pro}[{\cite{KL2}}]\label{p:cheeger}
Let $\Gr:=(\Er, \Vr)$ be a graph and $q:\Vr\to[0,\infty)$. Then, for all  $\Ur\subseteq \Vr$ we have on $ \Cc_{c}(\Ur)$.
\begin{align*}
\big(1-\sqrt{1-\al_{\Ur}^2}\big)(\deg+q) \leq\Delta+q\leq\big(1+\sqrt{1-\al_{\Ur}^2}\big) (\deg+q).
\end{align*}
\end{pro}



\subsection{Isoperimetric estimates at infinity}

Let the \emph{Cheeger constant at infinity} be defined as
\begin{align*}
\al_{\infty}=\sup_{\Kr\subseteq \Vr\,{\scriptsize{\mathrm{finite}}}}\al_{\Vr\setminus \Kr}.
\end{align*}
Clearly, $0\leq\al_{\Vr}\leq \al_\Ur\leq\al_{\infty}\leq 1$ for any
$\Ur\subseteq \Vr$.

As a consequence of Proposition \ref{p:cheeger}, we get the following theorem.

\begin{thm}\label{t:cheeger0}
Let $\Gr:=(\Er, \Vr)$ be a graph and $q:\Vr\to[0,\infty)$ be a
potential. Assume $\al_{\infty}>0$. Then, we have the following:
\begin{itemize}
  \item [(a)] For every
$\eps>0$ there is $k_{\eps}\geq 0$ such that on $ \Cc_{c}(\Vr)$
\begin{align*}
(1-\eps)\big(1-\sqrt{1-\al_{\infty}^2}\big)(\deg+q) -k_{\eps}&\leq\Delta+q\\
&\leq(1+\eps)\big(1+\sqrt{1-\al_{\infty}^2}\big) (\deg+q)+k_{\eps}.
\end{align*}
  \item [(b)] $\Dc((\Delta+q)^{1/2})=
\Dc((\deg+q)^{1/2})$.
  \item [(c)]  The operator $\Delta+q$ has
purely discrete spectrum if and only if we have $\liminf_{|x|\to\infty}(\mathrm{deg}+q)(x)=\infty$. In this case, if additionally $\al_{\infty}=1$, we get
\begin{align*}
    \liminf_{\lm\to\infty} \frac{\lm_{n}(\Delta+q)} {{\lm}_{n}(\mathrm{deg}+q)}=1.
\end{align*}
\end{itemize}
\end{thm}
\begin{proof}(a)
Let  $\eps>0$ and $\Kr\subseteq \Vr$ be finite and
large enough such that
\begin{align*}
(1-\eps)\big(1-\sqrt{1-\al_{\infty}^2}\big) &\leq \big(1- \sqrt{1-
  \al_{\Vr\setminus \Kr}^2}\big)
\\
 \big(1+ \sqrt{1-
  \al_{\Vr\setminus \Kr}^2}\big)&\leq (1+\varepsilon
)\big(1+\sqrt{1-\al_{\infty}^2}\big).
\end{align*}
From Proposition~\ref{p:cheeger} we conclude on $\cC_{c}(\Vr\setminus \Kr)$
\begin{align*}
(1-\eps)\big(1-\sqrt{1-\al_{\infty}^2}\big)(\deg+q) &\leq \big(1- \sqrt{1-\al_{\Vr\setminus \Kr}^2}\big)(\deg+q)\\
&\leq \Delta+q
  \leq \big(1+ \sqrt{1-  \al_{\Vr\setminus \Kr}^2}\big)(\deg+q)\\
  &\leq (1+\varepsilon
)\big(1+\sqrt{1-\al_{\infty}^2}\big)(\deg+q)
\end{align*}
By local finiteness the operators $\bone_{\Vr\setminus \Kr}(\Delta+q)\bone_{\Vr\setminus \Kr}$ and $\bone_{\Vr\setminus \Kr}(\deg+q)\bone_{\Vr\setminus \Kr}$ are bounded (indeed, finite rank) perturbations of $\Delta+q$ and $\deg+q$.  This gives rise to the constants $k_{\eps}$ and the inequality of (a) follows.
Now, (b) is an immediate consequence of (a), and (c)   follows by the Min-Max-Principle, Theorem~\ref{t:minmax}.
\end{proof}

\subsection{Relating sparseness and isoperimetric estimates.}
We now explain how the notions of sparseness and isoperimetric   estimates are exactly related.

First, we consider classical isoperimetric estimates.

\begin{thm}\label{t:sparse-isop-q0} Let $\Gr:=(\Vr,\Er)$ be a graph,     $a,k\geq 0$,   and let $q:\Vr\to [0,\infty)$ be a potential.
  \begin{itemize}
      \item [(a)]  $\al_{\Vr}\ge \displaystyle\frac{1}{1+a}$ if and only
       if  $(\Gr,q)$ is $(a,0)$-sparse.
    \item [(b)] If $(\Gr,q)$ is $(a,k)$-sparse, then
        \begin{align*}
            \al_{\Vr}\ge \frac{d-k}{d(1+a)},
        \end{align*}
        where $d:=\inf_{x\in\Vr}(\deg+q)(x)$. In particular,
        $\al_{\Vr}>0$ if $d>k$.

    \item [(c)] Suppose that $(\Gr,q)$ is $(a,k)$-sparse graph that is
      not $(a,k')$-sparse for all $k'<k$. Suppose also that there is
      $d$ such that $d=\deg(x)+q(x)$ for all $x\in\Vr$. Then
\[\al_{\Vr}=\displaystyle
      \frac{d-k}{d(1+a)}.\]
  \end{itemize}
  \end{thm}
\begin{proof}Let $\Wr\subset \Vr$ be a finite set. Recalling the
    identity      $\deg(\Wr)=2|\Er_{\Wr}|+|\partial\Wr|$
we notice that
\begin{align*}
\frac{1}{1+a}\leq \frac{|\partial\Wr|+q(\Wr)}{(\deg+q)(\Wr)}
\end{align*}
is equivalent to
\begin{align*}
 2 |\Er_{\Wr}| \leq a( |\partial\Wr|+q(\Wr))
\end{align*}
which proves (a).

For (b), the definition of $(a,k)$-sparseness yields
  \begin{align*}
    \frac{|\partial\Wr|+q(\Wr)}{(\deg+q)(\Wr)}=
    1-\frac{2|\Er_{\Wr}|}{(\deg+q)(\Wr)}\geq 1- a
    \frac{|\partial\Wr|+q(\Wr)}{(\deg+q)(\Wr)} - k
    \frac{|\Wr|}{(\deg+q)(\Wr)}.
  \end{align*}
This concludes immediately.

For (c), the lower bound of $\alpha_\Vr$ follows from (b). Since
  $(\Gr,q)$ is not $(a,k')$-sparse there is a finite $\Wr_0\subset
  \Vr$ such that
  \begin{align*}
    \frac{|\partial\Wr_0|+q(\Wr_0)}{(\deg+q)(\Wr_0)}
    < 1- a
    \frac{|\partial\Wr_0|+q(\Wr_0)}{(\deg+q)(\Wr_0)} - k'
    \frac{|\Wr_0|}{(\deg+q)(\Wr_0)}.
  \end{align*}
Therefore $\alpha_\Vr <(d-k')/(d(1+a))$. 
  \end{proof}

We next address the relation between almost sparseness and
    isoperimetry and show two ``almost equivalences''.

\begin{thm}\label{t:sparse-isop-q}
 Let $\Gr:=(\Vr,\Er)$ be a graph and let $q:\Vr\to [0,\infty)$ be a potential.
 \begin{itemize}
   \item [(a)] If $\alpha_\infty>0$, then $(\Gr,q)$ is $(a,k)$-sparse for some
$a>0$, $k\geq 0$.  On the other hand, if $(\Gr,q)$ is        $(a,k)$-sparse for some
$a>0$, $k\geq 0$ and
\[l:=\liminf_{|x|\to \infty} (\deg+q)(x) > k,\]
then
\begin{align*}\label{e:sparse-isop-q0}
\alpha_\infty \geq \frac{l-k}{l(a+1)}>0,
\end{align*}
if $l$ is finite and $\alpha_\infty\geq 1/(1+a)$ otherwise.
   \item [(b)]  If $\alpha_\infty=1$, then $(\Gr,q)$ is almost sparse. On the other hand, if $(\Gr,q)$ is almost sparse  and   $\liminf_{|x|\to\infty}(\deg+q)(x)=\infty$, then $\al_{\infty}=1$.
 \end{itemize}
 \end{thm}
\begin{proof}
The first implication of (a) follows from Theorem~\ref{t:cheeger0}~(a) and Theorem~\ref{t:form} (ii)$\Rightarrow$(i). For the opposite direction let $\eps>0$ and $\Kr\subseteq \Vr$ be finite such that $\deg+q\ge l-\eps$ on $\Vr\setminus \Kr$. Using the formula in the proof of  Theorem~\ref{t:sparse-isop-q0} above, yields for $\Wr\subseteq \Vr\setminus \Kr$
 \begin{align*}
    \frac{|\partial\Wr|+q(\Wr)}{(\deg+q)(\Wr)}&= 1-\frac{2|\Er_{\Wr}|}{(\deg+q)(\Wr)}\ge
    1-\frac{k|\Wr|+a(|\partial \Wr|+q(\Wr))}{(\deg+q)(\Wr)}\\
&\ge   1-\frac{k}{(l-\eps)} -\frac{a(|\partial\Wr|+q(\Wr))}{(\deg+q)(\Wr)}.
  \end{align*}
  This proves (a).\\
  The first implication of (b) follows from
  Theorem~\ref{t:cheeger0}~(a) and Theorem~\ref{t:q-sparse-deg}
  (ii)$\Rightarrow$(i). The other implication follows from (a) using
  the definition of almost sparseness.
\end{proof}

\begin{rem}
(a) We point out that without the assumptions on $(\deg+q)$ the converse implications  do not hold. For example the Cayley graph of $\Z$ is $2$-sparse (cf. Lemma~\ref{lem:planar=sparse}), but has $\alpha_\infty=0$.\\
(b) Observe that $\al_{\infty}=1$  implies        $\liminf_{|x|\to\infty}(\deg+q)(x)=\infty$. Hence,  (b)  can be rephrased  as the following equivalence: $\al_{\infty}=1$ is equivalent to  $(\Gr,q)$ almost sparse and $\liminf_{|x|\to\infty}(\deg+q)(x)=\infty$.\\
\end{rem}

{The previous theorems provides a slightly simplified proof of \cite{K1} which also appeared morally in somewhat different forms in \cite{D1,Woe}.
\begin{cor}Let $\Gr:=(\Vr,\Er)$ be a planar graph.
\begin{itemize}
  \item [(a)]  If for all vertices $\deg\geq   7$, then $\alpha_\Vr >0$.
\item [(b)]  If for all vertices  away from a finite set $\deg\geq   7$, then $\alpha_\infty >0$.
\end{itemize}
 \end{cor}
\proof Combine Theorem~\ref{t:sparse-isop-q0}  and Theorem~\ref{t:sparse-isop-q}  with Lemma
\ref{lem:planar=sparse}.\qed
}

\section{Examples}\label{s:sparse}

\subsection{Examples of  sparse graphs}

To start off, we exhibit two classes of sparse graphs. First we consider the case of graphs with bounded degree.
\begin{lem}\label{lem:bounded=sparse}
Let $\Gr:=(\Vr,\Er)$ be a graph. Assume $D:= \sup_{x\in \Vr} \deg(x)<+\infty$, then $\Gr$ is $D$-sparse.
\end{lem}
\begin{proof}
 Let $\Wr$ be a finite subset of $\Vr$. Then, $2|\Ec_\Wr|\leq \deg(\Wr) \leq D |\Wr|$.
\end{proof}

We turn to graphs which admit a $2$-cell embedding into $S_g$, where
$S_g$ denotes a compact orientable topological surface of genus
$g$. (The surface $S_{g}$ might be pictured as a sphere with $g$
handles.)  Admitting a $2$-cell embedding means that the graphs can be
embedded into $S_{g}$ without self-intersection.
By definition we say that a graph is planar when $g=0$. Note that unlike other possible definitions of planarity, we do not impose any local compactness on the embedding.

\begin{lem}\label{lem:planar=sparse}
\begin{itemize}
 \item[(a)] Trees are $2$-sparse.
 \item[(b)] Planar graphs are ${6}$-sparse.
\item[(c)] Graphs admitting a $2$-cell embedding into $S_g$ with $g\geq 1$ are
  ${4g+2}$-sparse.
\end{itemize}
\end{lem}
\begin{proof}
(a) Let $\Gr:=(\Vr,\Er)$ be a tree and  $\Gr_{\Wr}:=(\Wr,\Er_{\Wr})$ be a
finite  induced subgraph of $\Gr$. Clearly $|\Er_\Wr| \leq |\Wr|-1$. Therefore, every tree is $2$-sparse.

We treat the cases (b) and (c) simultaneously.
Let $\Gr:=(\Vr,\Er)$ be a graph which is connected $2$-cell
embedded in $S_g$ with $g\geq 0$ (as remarked above planar graphs correspond to $g=0$). Let $\Gr_{\Wr}:=(\Wr,\Er_{\Wr})$ be a
finite  induced subgraph of $\Gr$ which, clearly, also admits a $2$-cell
embedding  into $S_g$. The statement is  clear for $|\Wr|\leq 2$. Assume
$|\Wr|\geq 3$. Let $\Fr_{\Wr}$ be the faces induced by $\Gr_{\Wr}:=(\Wr,\Er_{\Wr})$ in $S_{g}$.
Here, all faces (even the outer one) contain at least $3$ edges, each edge belongs only to $2$ faces, thus,
\[
2|\Er_{\Wr}| \geq 3 |\Fr_{\Wr}|.
\]
Euler's formula,
$|\Wr|-|\Er_{\Wr}|+|\Fr_{\Wr}|=2-2g$,
gives then
\[
 2-2g+ |\Er_{\Wr}| =|\Wr|+|\Fr_{\Wr}|  \leq |\Wr| + \frac{2}{3} |\Er_{\Wr}|\]
 that is
\[|\Er_{\Wr}| \leq 3|\Wr| +6(g-1) \leq \max(2g+1,3) |\Wr|.\]
This concludes the proof.
\end{proof}

Next, we explain how to construct sparse graphs from
existing sparse graphs.

\begin{lem}\label{l:sums_of_graphs}
 Let $\Gr_1:=(\Vr_{1},\Er_1)$ and $\Gr_2:=(\Vr_{2},\Er_2)$ be two graphs.
 \begin{itemize}
   \item [(a)] Assume $\Vr_{1}=\Vr_{2}$, $\Gr_1$ is $k_1$-sparse and  $\Gr_2$
 is $k_2$-sparse. Then, $\Gr:=(\Vr,\Er)$ with $\Er:=\max(\Er_1, \Er_2)$ is $(k_1+k_2)$-sparse.
   \item [(b)]  Assume $\Gr_1$ is $k_1$-sparse and  $\Gr_2$
 is $k_2$-sparse. Then $\mathcal{G}_1\oplus \mathcal{G}_2:=(\Vr, \Er)$ with where $\Vr:=\Vr_1\times \Vr_2$ and
\[ \Er\left((x_{1},x_{2}),
  (y_1,y_2)\right):= \delta_{\{x_{1}\}}(y_1)\cdot\Er_2(x_{2},y_2) + \delta_{\{x_{2}\}}(y_2)\cdot\Er_1(x_{1},y_1),\]
  is $(k_1+k_2)$-sparse.
   \item [(c)] Assume $\Vr_{1}=\Vr_{2}$, $\Gr_1$ is $k$-sparse and $\Er_{2}\leq \Er_{1}$. Then, $\Gr_{2}$ is $k$-sparse.
 \end{itemize}
\end{lem}
\begin{proof}For (a) let $\Wr\subseteq \Vr$ be finite and note that $|\Er_{\Wr}|\leq|\Er_{1,\Wr}|+|\Er_{2,\Wr}|$. For (b) let $p_{1}$, $p_{2}$ the canonical projections from $\Vr$ to $\Vr_{1}$ and $\Vr_{2}$. For finite $\Wr\subseteq \Vr$ we observe $$|\Er_{\Wr}|=|\Er_{1,p_{1}(\Wr)}|+|\Er_{2,p_{2}(\Wr)}| \leq k_{1}|p_{2}(\Wr)|+ k_{2}|p_{1}(\Wr)|\leq(k_{1}+k_{2})|\Wr|.$$
For (c) and $\Wr\subseteq \Vr$ finite, we have
$|\Er_{2,\Wr}|\leq|\Er_{1,\Wr}|$ which yields the statement.
\end{proof}

\begin{rem}
(a) We point out that there are bi-partite graphs which are  not
sparse. See for example \cite[Proposition~4.11]{Gol2} or take an
antitree, confer \cite[Section~6]{KLW},  where the number of vertices
in the spheres grows monotonously to $\infty$.
\\
(b) The last point of the lemma states that the $k$-sparseness  is
non-decreasing when we remove  edges from the graph. This is not the
case for the isoperimetric constant.
\end{rem}

\subsection{Examples of almost-sparse and {$(a,k)$}-sparse graph}
We construct a series of examples which are perturbations of a radial tree. They illustrate that sparseness, almost sparseness and $(a,k)$-sparseness are indeed  different concepts.

Let   $\beta=(\beta_{n})$, $\gamma=(\gm_{n})$ be two sequences of natural numbers.
Let $\Tr=\Tr(\be)$ with $\Tr=(\Vr,\Er^\Tr)$ be a radial tree with root $o$ and
vertex degree $\be_{n}$ at the $n$-th sphere,
that is every vertex which has natural graph distance $n$ to $o$ has
$(\be_{n}-1)$ forward neighbors. We denote the distance spheres by
$S_{n}$. We let $\Gr(\be,\gm)$ be the set of graphs
$\Gr:=(\Vr,\Er^{\Gr})$ that are super graphs of $\Tr$ such that the
induced subgraphs $\Gr_{S_{n}}$ are $\gm_{n}$-regular and
$\Er^{\Gr}(x,y)=\Er^{\Tr}(x,y)$ for $x\in S_{n},y\in S_{m}$, $m\neq
n$.

Observe that $\Gr(\be,\gm)$ is non empty {if and only if}
$\gm_{n}\prod_{j=0}^{n}(\be_{j}-1)$ is even and
$\gm_{n}<|S_{n}|=\prod_{j=0}^{n}(\be_{j}-1)$ for all $n\ge0$.

\begin{figure}[!ht]
\scalebox{1.6}[0.6]{\includegraphics{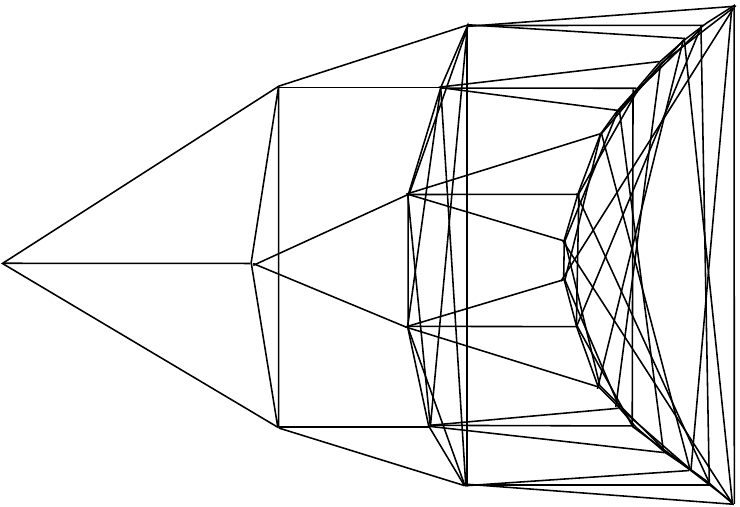}}
\caption{\label{f:example} $\Gr$ with $\be=(3,3,4,\ldots)$ and $\gm=(0,2,4,5,\ldots)$.}
\end{figure}

\begin{pro}
Let $\be,\gm\in\N_{0}^{\N_{0}}$, $a=\limsup_{n\to\infty}{\gm_{n}}/{\be_{n}}$ and
$\Gr\in\Gr(\be,\gm)$.
\begin{itemize}
  \item [(a)] If $a=0$, then $\Gr$ is almost sparse. The graph $\Gr$ is sparse if and only if $\limsup_{n\to\infty}\gm_{n}<\infty $.
  \item [(b)] If $a>0$, then $\Gr$ is $(a',k)$-sparse  for some $k\ge0$ if $a'>a$. Conversely, if $\Gr$ is $(a',k)$-sparse  for some $k\ge0$, then $a'\geq a$.
\end{itemize}
\end{pro}
\begin{proof}
Let $\eps>0$ and let $N\ge0$ be so large that
\begin{align*}
    \gm_{n}\leq (a+\eps)\be_{n},\quad n\ge N.
\end{align*}
Set $C_\varepsilon :=\sum_{n=0}^{N-1}\deg^\Gr(S_{n})$.
Let $\Wr$ be a non-empty finite subset of
$\Vr$. We calculate
 \begin{align*}
 2 |\Er^\Gr_{\Wr}| + |\partial^\Gr \Wr|
 &=  \deg^\Gr(\Wr)=\deg^\Tr(\Wr) + \sum_{n\geq 0}|\Wr\cap S_n| \gm_n\\
 & \leq  \deg^\Tr(\Wr) + (a+\ve)\sum_{n\geq 0}|\Wr\cap S_n| \be_{n} + \sum_{n=0}^{N-1}|\Wr\cap S_n|\gm_n \\
 &\leq (1+a+\ve) \deg^\Tr(\Wr) +  C_\varepsilon |\Wr|\\
&= 2(1+a+\ve) |\Er^\Tr_{\Wr}| + (1+a+\ve) |\partial^\Tr \Wr|+ C_\ve |\Wr| \\
&\leq (2(1+a+ \ve)+ C_\ve) |\Wr| + (1+a+\ve) |\partial^\Tr \Wr|,
\end{align*}
where we used that trees are
$2$-sparse in the last inequality. Finally, since $ |\partial^\Gr \Wr| \geq |\partial^\Tr
\Wr|$, we conclude
\[
 2|\Er^\Gr_{\Wr}| \leq \left(2(1+a+\ve)+ C_\ve\right) |\Wr| +
 (a+\ve)|\partial^\Gr \Wr|.
\]
This shows that the graph in (a) with $a=0$ is almost sparse and that the graph in (b) with $a>0$ is $(a+\eps,k_{\eps})$-sparse for $\eps>0$ and $k_{\eps}=2(1+a+\ve)+ C_\ve$.  Moreover, for the other statement of  (a) let $k_{0}=\limsup_{n\to\infty}\gm_{n} $ and note that for $\Gr_{S_{n}}$
\begin{align*}
    2|\Er_{S_{n}}|= \gm_{n} |S_{n}|.
\end{align*}
Hence, if  $k_{0}=\infty$, then $\Gr$ is not sparse. On the other hand, if $k_{0}<\infty$, then $\Gr$ is $(k_{0}+2)$-sparse by
 Lemma~\ref{l:sums_of_graphs} as $\Tr$ is 2-sparse by Lemma~\ref{lem:planar=sparse}. This finishes the proof of (a). Finally, assume that $\Gr$ is $(a',k)$-sparse with $k\ge0$. Then, for $\Wr=S_{n}$
 \begin{align*}
    \gm_{n}|S_{n}|=2|\Er_{S_{n}}|\leq k|S_{n}|+a'|\partial^{\Gr} S_{n}|=k|S_{n}|+a'\be_{n}|S_{n}|
 \end{align*}
 Dividing by $\beta_{n}|S_{n}|$ and taking the limit yields $a\leq a'$.  This proves (b).
\end{proof}

\begin{rem}
In (a), we may suppose alternatively that we have the complete graph
on $S_n$ and the following exponential growth $\lim_{n\to \infty}
\frac{|S_{n}|}{|S_{n+1}|}=0$.
\end{rem}


\appendix
\section{Some general operator theory}
We collect some consequences  of standard results from functional analysis that are used in the paper. Let $H$ be a Hilbert space with norm $\|\cdot\|$. For a quadratic form $Q$, denote the form norm by $\|\cdot\|_{Q}:=\sqrt{Q(\cdot)+\|\cdot\|^{2}}$. The following  is a direct consequence of the Closed Graph Theorem, (confer e.g. \cite[Satz 4.7]{We}).

\begin{thm}\label{t:imt} Let $(Q_{1},\Dc(Q_{1}))$ and $(Q_{2},\Dc(Q_{2}))$ be closed non-negative quadratic forms with a common form core $\Dc_{0}$.
Then, the following are equivalent:
\begin{itemize}
  \item [(i)]  $\Dc(Q_{1})\leq\Dc(Q_{2})$.
  \item [(ii)] There are constants $c_{1}>0$, $c_{2}\ge0$ such that $c_{1}Q_{2}-c_{2}\leq Q_{1}$ on $D_{0}$.
\end{itemize}
\end{thm}
\begin{proof}If (ii) holds, then any $\|\cdot\|_{Q_{1}}$-Cauchy sequence is a $\|\cdot\|_{Q_{2}}$-Cauchy sequence. Thus, (ii) implies (i). On the other hand,  consider the identity map $j:(\Dc(Q_{1}),\|\cdot\|_{Q_{1}})\to (\Dc(Q_{2}),\|\cdot\|_{Q_{2}})$. The map $j$ is closed as it is defined on the whole Hilbert space $(\Dc(Q_{1}),\|\cdot\|_{Q_{1}})$ and, thus, bounded by the Closed Graph Theorem  \cite[Theorem~III.12]{RS} which implies (i).
\end{proof}

For a selfadjoint operator $T$ which is bounded from below, we denote the
bottom of the  spectrum by $\lm_{0}(T)$ and the bottom of the essential
spectrum by $\lm_{0}^{\mathrm{ess}}(T)$. Let $n(T)\in \N_{0}\cup\{\infty\}$ be the dimension of the range of the spectral projection of $(-\infty,\lm_{0}^{\mathrm{ess}}(T))$.  For $\lm_{0}(T)<\lm_{0}^{\mathrm{ess}}(T)$ we denote the eigenvalues below $\lm_{0}^{\mathrm{ess}}(T)$  by $\lm_{n}(T)$, for $0  \leq n \leq n(T)$,  in increasing  order counted with multiplicity.

\begin{thm}\label{t:minmax} Let $(Q_{1},\Dc(Q_{1}))$ and $(Q_{2},\Dc(Q_{2}))$ be closed non-negative quadratic forms with a common form core $\Dc_{0}$  and let $T_{1}$ and $T_{2}$ be the corresponding selfadjoint operators.
Assume there are constants $c_{1}>0$, $c_{2}\in\R$ such that on $\Dc_{0}$
\begin{align*}
    c_{1}Q_{2}-c_{2}\leq Q_{1}.
\end{align*}
Then,  $c_{1}\lm_{n}(T_{2})-c_{2}\leq \lm_{n}(T_{1})$, for $0 \leq n
\leq \min(n(T_1), n(T_2))$. Moreover,
$c_{1}\lm_{0}^{\mathrm{ess}}(T_{2})-c_{2}\leq
\lm_{0}^{\mathrm{ess}}(T_{1})$, in particular,
$\si_{\mathrm{ess}}(T_{1})=\emptyset$ if
$\si_{\mathrm{ess}}(T_{2})=\emptyset$ and in this case
      \begin{align*}
        c_{1}\leq\liminf_{n\to\infty} \frac{\lm_{n}(T_{1})}{\lm_{n}(T_{2})}.
      \end{align*}
\end{thm}
\begin{proof}
Letting
\begin{align*}
    \mu_{n}(T)=\sup_{\ph_{1},\ldots,\ph_{n}\in H}\inf_{0\neq\psi\in\{\ph_{1},\ldots,\ph_{n}\}^{\perp}\cap \Dc_0}\frac{\langle T\psi,\psi\rangle}{\langle \psi,\psi\rangle},
\end{align*}
for a selfadjoint operator $T$, we know by the Min-Max-Principle \cite[Chapter XIII.1]{RS} $\mu_{n}(T)=\lm_{n}(T)$ if $\lm_{n}(T)<\lm_{0}^{\mathrm{ess}}(T)$
and $\mu_{n}(T)=\lm_{0}^{\mathrm{ess}}(T)$ otherwise, $n\ge0$.
Assume $n\leq\min\{n(T_{1}),n(T_{2})\}$ and  let $\ph_{0}^{(j)},\ldots,\ph_{n}^{(j)}$ be the eigenfunctions of $T_{j}$ to $\lm_{0}(T_{j}),\ldots,\lm_{n}(T_{j})$ we get
\begin{align*}
c_{1}\lm_{n}(T_{2})-c_{3} &=\inf_{0\neq\psi\in\{\ph_{1}^{(2)},\ldots,\ph_{n}^{(2)}\}^{\perp}\cap \Dc_0}\Big(c_{1}\frac{\langle T_{2}\psi,\psi\rangle}{\langle \psi,\psi\rangle}-c_{3}\Big)\\
&\leq\inf_{0\neq\psi\in\{\ph_{1}^{(2)},\ldots,\ph_{n}^{(2)}\}^{\perp}\cap \Dc_0} \frac{\langle T_{1}\psi,\psi\rangle}{\langle \psi,\psi\rangle}\leq\mu_{n}(T_{1})=\lm_{n}(T_{1})
\end{align*}
This
directly implies the first statement. By a similar argument the  statement about the bottom of the essential spectrum follows, in particular, $\lm_{0}^{\mathrm{ess}}(T_{2})=\infty$ implies $\lim_{n\to\infty}\mu_{n}(T_{1})=\infty$ and, thus,
$\lm_{0}^{\mathrm{ess}}(T_{1})=\infty$. In this
case  $\lm_{n}(T_{2})\to\infty$,
$n\to\infty$, which implies the final statement.
\end{proof}

Finally, we give a lemma which helps us to transform inequalities under form perturbations.

\begin{lemma}\label{l:formperturbation}
Let $(Q_{1},\Dc(Q_{1}))$, $(Q_{2},\Dc(Q_{2}))$ and $(q,\Dc(q))$ be
closed symmetric non-negative quadratic forms with a common form core
$\Dc_{0}$ such that there are $\al\in (0,1)$, $C_{\al}\ge0$ such that
\[q\le
\al Q_{1}+C\]
 on $\Dc_{0}$. If for
$a\in(0,1)$ and $k\ge0$
$$(1-a)Q_{2}-k\leq Q_{1}\quad\mbox{ on }\Dc_{0},$$
 then
 \begin{align*}
         \frac{(1-\al)(1-a)}{(1-\al(1-a))}(Q_{2}-q) -\frac{(1-\al)k+aC_\al}{(1-\al(1-a))}\leq Q_{1}-q,\qquad\mbox{on $\Dc_{0}$}.
\end{align*}
In particular, if $a\to 0^+$, then $(1-\al)(1-a)/(1-\al(1-a))\to
1^-$ and if $\al\to 0^+$, then $(1-\al)(1-a)/(1-\al(1-a))\to (1-a)$.
\end{lemma}
\begin{proof} The assumption on $q$ implies
\begin{align*}
    q\leq \frac{\al}{(1-\al)}(Q_{1}-q)+\frac{{C_\al}}{(1-\al)}.
\end{align*}
We subtract $(1-a)q$ on each side of the lower bound in $(1-a)Q_{2}-k\leq Q_{1}$. Then, we get
\begin{align*}
  (1-a)(Q_{2}-q)-k\leq (Q_{1}-q)+a q\leq
  \frac{1-\al(1-a)}{(1-\al)}(Q_{1}-q)+\frac{a C_\al}{(1-\al)}
\end{align*}
and, thus, the asserted inequality follows.
\end{proof}

\textbf{Acknowledgement}.
MB was partially supported by the ANR project HAB
(ANR-12-BS01-0013-02). SG was partially supported by the ANR project
GeRaSic and SQFT. MK enjoyed the hospitality of Bordeaux University
when this work started. Moreover, MK  acknowledges the financial
support of the German Science Foundation (DFG), Golda Meir Fellowship,
the Israel Science Foundation (grant no. 1105/10 and  no. 225/10)  and
BSF grant no. 2010214.

\end{document}